\documentclass[12pt]{amsart}
\usepackage{amsmath,amssymb,latexsym}
\usepackage{fullpage}
\usepackage{amscd}
\theoremstyle{plain}

\newtheorem{assumption}{Assumption}
\newtheorem{lemma}{Lemma}
\newtheorem{definition}{Definition}

\def\A{\mathbb{A}}

\def\Ind{\text{Ind}}

\def\Hom{\text{Hom}}

\begin{document}

\title{Dimensions of Automorphic Representations, $L$-Functions and Liftings}
\author{Solomon Friedberg and David Ginzburg}

\address{Friedberg:  Department of Mathematics, Boston College, Chestnut Hill, MA 02467-3806, USA}
\email{solomon.friedberg@bc.edu}
\address{Giinzburg: School of Mathematical Sciences, Tel Aviv University, Ramat Aviv, Tel Aviv 6997801, Israel}
\email{ginzburg@post.tau.ac.il}
\thanks{This research was supported by the US-Israel Binational Science foundation, grant number 2016000, and by the NSF, grant number DMS-1801497 (Friedberg).}
\begin{abstract}
There are many Rankin-Selberg integrals representing Langlands $L$-functions, and it is not apparent what the limits of the Rankin-Selberg
method are.  The Dimension Equation is an equality satisfied by many such integrals that suggests a priority for further investigations.
However there are also Rankin-Selberg integrals that do not satisfy this equation.  Here we propose an extension and reformulation
of the dimension equation that includes many additional cases.  We explain some of these cases, including the new doubling integrals
of the authors, Cai and Kaplan.  We then show how this same
equation can be used to understand theta liftings, and how doubling integrals fit into a lifting framework.   We give an example of a new type of
lift that is natural from this point of view.
\end{abstract}
\maketitle

\section{Introduction}
Two broad classes of integrals appear frequently in the theory of automorphic forms.  
Let $G$ be a reductive group defined over a number field $F$, $\rho$ be a complex analytic representation of the $L$-group of $G$, and $\pi$ be
an irreducible automorphic representation of $G(\A)$.  First,
one may sometimes represent the Langlands $L$-function $L(s,\pi,\rho)$ (for $\Re(s)\gg0$) as an integral, and the desired analytic properties 
of this $L$-function may then be deduced from the integral representation.  Such constructions, often called Rankin-Selberg integrals, have a long history with many examples (Google Scholar
lists 3,370 results for the phrase ``Rankin-Selberg integral"), with Eisenstein
series and their Fourier coefficients appearing in many of the integrals.  See Bump \cite{Bump-2005} for an engaging survey.
Second, given two reductive groups $G,H$ there is sometimes 
a function $\theta(g,h)$ on $G(\A)\times H(\A)$ that is $G(F)\times H(F)$ invariant and that may be used as an integral kernel
to transport automorphic representations on $G(\A)$ to $H(\A)$.  One considers the functions
$$f_\varphi(h):=\int_{G(F)\backslash G(\A)} \varphi(g)\,\theta(g,h)\,dg$$
as $\varphi$ runs over the functions in $\pi$ and lets $\sigma$ be the representation of $H(F)\backslash H(\A)$ generated by the functions $f_\varphi$.  The most familiar example is the classical theta correspondence, where the integral kernel $\theta$ is constructed from the Weil representation (see Gan \cite{Gan-ICM} for recent progress), while other classes of such integrals
are given by the authors and Bump \cite{BFG-2006} and by Leslie \cite{Leslie}.

We pose two natural questions.  First, the construction of Rankin-Selberg integrals representing $L$-functions is often quite involved, and it may take a great deal of work
to see that a specific integral is Eulerian. Is there any commonality among the known Rankin-Selberg integrals that can be used
to decide whether a specific integral is a worthy candidate for investigation, or to say it differently, to rule out
integrals as being unlikely to represent an $L$-function?  Second, is there any way to know whether an automorphic function of two variables $\theta(g,h)$ is likely to be a
useful kernel function, and if so, can one predict the properties of the representation $\sigma$ from those of $\pi$ and $\theta$?  For example, given $\pi$ and $\theta$,
when is it reasonable to think that $\sigma$ might be generic?

An answer to the first question was proposed by the second named author in \cite{Ginzburg-2014}: the {\sl dimension equation} describes an equality
between the dimensions of groups and the dimensions of the representations that is satisfied by many integrals that represent $L$-functions. 
Below we develop the dimension question in some detail and illustrate it in many cases, including cases of integrals that represent the product of two or more distinct
Langlands $L$-functions in separate complex variables.  In fact a refinement of the dimension equation expands its applicability to additional cases.
We shall explain this refinement, and show how this
allows us to include doubling integrals including the new doubling integrals of the authors, Cai and Kaplan \cite{CFGK}.   Then we shall use the dimension equation to discuss integral kernels, 
showing how the equation gives an indication of what to expect in integral kernel constructions, and 
explain how doubling integrals may be used to bridge these two classes of constructions.  Last, we shall pursue the dimension equation farther in specific cases, providing new examples
for further study.

To conclude our introduction, we describe the contents by section.
In Section~\ref{ginzequation} we introduce unipotent orbits and state the dimension equation, following~ \cite{Ginzburg-2014}.  
Though this paper may be read independently from \cite{Ginzburg-2014}, it is a natural continuation of that work.  Then, in Section~\ref{dim-eq-examples} we give many examples of Rankin-Selberg integrals
that satisfy the dimension equation.  We conclude the section by presenting an exotic example---an integral of Rankin-Selberg type that is Eulerian by two different choices of Eisenstein series, one satisfying the dimension equation
but the other not.  This motivates the need to extend the dimension equation.  This extension is described in Section~\ref{extending-dc}, and examples
of the extended dimension equation are presented in Section~\ref{examples-of-extend}.  Next,  in Section~\ref{integral-kernels} integral kernels are connected to the dimension equation,
and the use of the dimension equation to predict aspects of the resulting correspondence is illustrated. Section~\ref{general} revisits doubling integrals from the perspective of integral
kernels, and formulates a general classification question.  Then in Section~\ref{example} a new low rank example is presented and its analysis is described in brief.
The global theory has a local counterpart, and the final Section illustrates the local properties that appear in this context.

Part of the preparation of this paper occurred when the first-named author was a visitor at the Simons Center for Geometry and Physics, and he expresses his appreciation for this opportunity.
While the paper was in the final stages of preparation, Aaron Pollack communicated to us that Shrenik Shah has independently suggested an extension of the dimension equation.

\section{The Dimension Equation}\label{ginzequation}

Fix a number field $F$ and let $\A$ be its ring of adeles.  If $G$ is an algebraic group defined over $F$, then we write $[G]$ for the quotient $G(F)\backslash G(\A)$.

We begin by recalling some facts about unipotent orbits.  References for this material are Collingwood and McGovern \cite{C-M} and Ginzburg \cite{G1}.
Let $G$ be a reductive group defined over $F$, $\overline{F}$ be an algebraic closure of $F$, and let $\mathcal{O}$ be a unipotent orbit of $G(\overline{F})$.  If $G\subset GL(V)$ is a classical
group then these orbits are indexed by certain partitions of $\dim(V)$.  For example, if $G=GL_n$ then they are indexed by all partitions of $n$, while if 
$G=Sp_{2n}$ then they are indexed by the partitions of $2n$ such that each odd part occurs an even 
number of times.  When $\mathcal{O}$ is indexed by a partition $\mathcal{P}$, we simply write $\mathcal{O}=\mathcal{P}$.

For convenience suppose that $G$ is a split classical group.
 Fix a Borel subgroup $B=TN$ with unipotent radical $N$, let $\Phi$ denote the positive roots with respect to $B$, and for $\alpha\in\Phi$ let $x_\alpha(t)$ denote the corresponding one-parameter 
subgroup of $N$.
Then one attaches a unipotent subgroup $U_\mathcal{O}\subseteq N$ to $\mathcal{O}$.  
This group may be described as follows.  If $\mathcal{O}$ is given in partition form by
$\mathcal{O}=(p_1^{e_1}\dots p_k^{e_k})$ (we show repeated terms in a partition using exponential notation) let $h_\mathcal{O}(t)$ be the diagonal matrix whose entries are  $\{t^{p_i-2j-1}\mid 0\leq j\leq p_i-1\}$ repeated $e_i$ times, with the entries 
arranged in non-increasing order in terms of power of $t$.  This gives rise to a filtration $N\supset N_1\supset N_2 \supset \cdots$ of $N$, where $N_i=N_{i,\mathcal{O}}$ is given by
\begin{equation}\label{filtration}
N_{i,\mathcal{O}}(F)=\{x_\alpha(r)\in N(F)\mid  \alpha\in\Phi~\text{and}~h_\mathcal{O}(t)x_\alpha(r)h_\mathcal{O}(t)^{-1}=x_\alpha(t^jr)~\text{for some $j$ with~} j\geq i\}.
\end{equation}
 Also, let
$G_{\mathcal{O}}$ be the stabilizer of $h_\mathcal{O}(t)$ in $G$.

Fix a nontrivial additive character $\psi$ of $F\backslash \A$, and let $L_{2,\mathcal{O}}=N_{2,\mathcal{O}}/[N_{2,\mathcal{O}},N_{2,\mathcal{O}}]$.  
Then the characters of the abelian group $L_{2,\mathcal{O}}(\A)$ may be identified with $L_{2,\mathcal{O}}(F)$.   Also $G_{\mathcal{O}}(F)$ acts on the characters of $L_{2,\mathcal{O}}(\A)$ and hence on $L_{2,\mathcal{O}}(F)$  by conjugation.
Over the algebraically closed field $\overline{F}$, the action of $G_{\mathcal{O}}(\overline{F})$ on $L_{2,\mathcal{O}}(\overline{F})$ has an open orbit, and the stabilizer in $G_\mathcal{O}(\overline{F})$
of a representative for this orbit is a reductive group whose Cartan type is uniquely determined by the orbit.  A character $\psi_{\mathcal{O}}$ of $L_{2,\mathcal{O}}(F)$ will be called a generic 
character associated to $\mathcal{O}$ if
the connected component of its stabilizer in $G_{\mathcal{O}}(F)$ has, after base change to $\overline{F}$, the same Cartan type.  We caution the reader that there may be infinitely many 
$G_{\mathcal{O}}(F)$-orbits of generic characters associated to a single $\mathcal{O}$.  (For an example, see \cite{Ginzburg-2014}, p.\ 162.)

Suppose that $G$ is a reductive group, $\mathcal{O}$ is a unipotent orbit of $G$, and $\psi_{\mathcal{O}}$ is a generic character. Let $U_{\mathcal{O}}=N_{2,\mathcal{O}}$ and regard 
$\psi_{\mathcal{O}}$ as a character of $U_{\mathcal{O}}(\A)$ in the canonical way.
 If $\varphi$ is an automorphic form on $G(\A)$, its {\sl Fourier coefficient 
with respect to $\mathcal{O}$} is defined to be
$$\varphi^{U_\mathcal{O},\psi_{\mathcal{O}}}(g)=\int_{[U_{\mathcal{O}}]}\varphi(ug)\psi_{\mathcal{O}}(u)\,du,\qquad g\in G(\A). $$
If $\pi$ is an automorphic representation of $G(\A)$, we say that $\pi$ has {\sl nonzero Fourier coefficients with respect to $\mathcal{O}$} if the set of functions $\varphi^{U_\mathcal{O},\psi_{\mathcal{O}}}(g)$
is not identically zero as $\varphi$ runs over the space of $\pi$ and $\psi_{\mathcal{O}}$ runs over set of generic characters associated to $\mathcal{O}$.

We recall that the set of unipotent orbits has a natural partial ordering, which corresponds to inclusion of Zariski closures and corresponds to the dominance order for the associated partitions.
There is a unique maximal unipotent orbit $\mathcal{O}_{\text{max}}$ and for this orbit the group $U_{\mathcal{O}_\text{max}}$ is $N$,the unipotent
radical of the Borel subgroup.
If $\pi$ is an automorphic representation of $G(\A)$ we let $\mathcal{O}(\pi)$ denote the set of maximal unipotent orbits for which $\pi$ has nonzero Fourier coefficients.  For example,
the automorphic representation $\pi$ is generic if and only if $\mathcal{O}(\pi)=\{\mathcal{O}_{\text{max}}\}$.

We remark that these same definitions apply without change in the covering group case.  Indeed, if $\widetilde{G}$ is a cover of $G(\A)$, then $N(\A)$ embeds canonically in $\widetilde{G}(\A)$
and so the same definitions apply.  For example, in \cite{BFG-theta} the authors construct a representation $\Theta_{2n+1}$ of the double cover of $SO_{2n+1}(\A)$ for $n\geq4$ such
that $\mathcal{O}(\Theta_{2n+1})=(2^{2m}1^{2j+1})$, where $n=2m+j$ with $j=0$ or $1$. Here and below we omit the set notation when the set 
$\mathcal{O}(\pi)$ is a singleton and identify $\mathcal{O}(\pi)$ with the partition.

Recall that each nilpotent orbit in a Lie algebra has a dimension.  In fact, these dimensions are computed for classical groups in terms of partitions in \cite{C-M}, Corollary 6.1.4.
We use this to discuss and work with the dimensions of the unipotent orbits under consideration here.
To illustrate, on the symplectic group $Sp_{2n}$, let ${\mathcal O}$ be the unipotent orbit associated to the partition $(n_1n_2\ldots n_r)$ of $2n$. 
For such a partition to be a symplectic partition it is required that all odd numbers in the partition occur with even multiplicity.  Then,  
\begin{equation}\label{dim2}
\dim{\mathcal O}=2n^2+n-\frac{1}{2}\sum_{i=1}^r(2i-1)n_i-\frac{1}{2}a,
\end{equation}
where $a$ is the number of odd integers in the partition $(n_1n_2\ldots n_r)$. 
Importantly, there is also a connection between the dimension of an orbit $\mathcal{O}$ and the Fourier coefficients above.  That is:
\begin{equation}\label{compute-FJ}\tfrac12\dim{\mathcal O}=\dim N_{2,\mathcal{O}} + \tfrac12 \dim N_{1,\mathcal{O}}/N_{2,\mathcal{O}}.
\end{equation}
As this suggests, when $N_{1,\mathcal{O}}\neq N_{2,\mathcal{O}}$, there is another natural coefficient of Fourier-Jacobi type (that is, involving a theta series);
see \cite{Ginzburg-2014},  equation (7), for details, and the discussion of \eqref{func1} at the end of Section~\ref{extending-dc} below.

Throughout the rest of this paper we make the following assumption:
\begin{assumption}\label{unique-orbit} For all automorphic representations $\pi$ under consideration, the dimension of each orbit in the set $\mathcal{O}(\pi)$ is the same.
\end{assumption}
We know of no examples where Assumption~\ref{unique-orbit} does not hold, and it is expected that it is always true (\cite{G1}, Conjecture 5.10).  In fact, we know of no examples where $\mathcal{O}(\pi)$ is not a singleton.
Under Assumption~\ref{unique-orbit}, we write $\dim\mathcal{O}(\pi)$ for the dimension of any orbit in the set $\mathcal{O}(\pi)$. 
Given a representation for which Assumption~\ref{unique-orbit} holds, define (following \cite{G1}, Definition~5.15) the Gelfand-Kirillov dimension of $\pi$: 
$$\dim(\pi)=\tfrac12 \dim\mathcal{O(\pi)},$$ 
(compare \cite{B-V}, (3.2.6) and \cite{K} (2.4.2), Remark (iii)).  By definition, if $\chi$ is an idele class character then $\dim(\chi)=0$.  Similarly if $\psi$ is an additive
character of $[U]$ where $U$ is a unipotent group then $\dim(\psi)=0$.

It is expected that it is possible to compute the unipotent orbit attached to an Eisenstein series from the unipotent orbits of its inducing data.  See \cite{Ginzburg-2014}, Section~4.3.
This has been confirmed in the degenerate case \cite{Cai}.  More generally, suppose that $P$ is a parabolic subgroup of $G$ with Levi decomposition $P=MN$, and $\tau$ is an automorphic
representation of $M(\A)$ such that Assumption 1 holds for $\tau$, and consider the Eisenstein series $E_\tau(g,s)$ on $G(\A)$ corresponding to the induced representation 
$\Ind_{P(\A)}^{G(\A)} 
\tau \delta_P^s.$ Then one expects that
\begin{equation}\label{dim-of-ES}\dim E_\tau(g,s)=\dim \tau + \dim U.
\end{equation}
This formula can be checked in many cases (to give one example, if $\tau$ is generic then $E_\tau$ is too, and so \eqref{dim-of-ES} holds)
and we shall make use of it below. 

We now have the information needed to explain the Dimension Equation of \cite{Ginzburg-2014}, Definition 3.
Suppose one has a Rankin-Selberg integral over the various groups $G_j$
and involving various automorphic representations $\pi_i$, and the integral unfolds to unique functionals which are factorizable.
Here both the adelic modulo rational quotients in the integral and the automorphic representations are with respect to the same field $F$ and ring of adeles $\A$.
Then the expectation formulated in \cite{Ginzburg-2014} is 
that the sum of the dimensions of the representations is (generally) equal to the sum of the dimensions of the groups.  
\begin{definition} The Dimension Equation is the equality
$$\sum_i \dim \pi_i =\sum_j \dim G_j.$$
\end{definition}

A number of examples of this equation in the context of Rankin-Selberg integrals are presented in \cite{Ginzburg-2014}.   To clarify the extent of applicability of this equation, we will give additional examples in the next Section.

\section{Examples of the Dimension Equation}\label{dim-eq-examples}

We begin with classical examples of the dimension equation.  Riemann's second proof of the analytic continuation of the zeta function represents it as a 
Mellin transform of the Jacobi theta function, and this integral satisfies the dimension equation (both sides are 1).  Similarly, the Hecke integral
\begin{equation}\label{Hecke}\int_{[\mathbb{G}_m]} \varphi\begin{pmatrix}a&\cr&1\end{pmatrix} |a|^s\,d^\times a\end{equation}
satisfies the equation (both sides are again 1), and the classical Rankin-Selberg integral 
\begin{equation}\label{classicalRS}\int_{[PGL_2]} \varphi_1(g)\varphi_2(g) E(g,s)\,dg\end{equation}
has group of dimension 3 and three representations that are each of dimension 1.  

More generally, the Rankin-Selberg integrals of Jacquet, Piatetski-Shapiro and Shalika \cite{J-PS-Sh} on $GL_n\times GL_k$ satisfy the dimension equation.  Suppose that
$\pi_1$, $\pi_2$ are irreducible cuspidal automorphic representations of $GL_n(\A)$, $GL_k(\A)$ resp.
Since all cuspidal automorphic representations of general linear groups are generic, we have $\dim(\pi_1)=n(n-1)/2$ and $\dim(\pi_2)=k(k-1)/2$.
Let $\phi_i\in \pi_i$ for $i=1,2$.

If $k=n$ then the Ranklin-Selberg integral representing $L(s,\pi_1\times \pi_2)$, generalizing \eqref{classicalRS},  is given by
\begin{equation}\label{equalRS}\int_{[PGL_n]} \varphi_1(g)\varphi_2(g) E(g,s)\,dg\end{equation}
where $E$ is the (``mirabolic") Eisenstein series induced from the standard parabolic subgroup  $P$ of type $(n-1,1)$ and character $\delta_P^s\eta^{-1}$, where 
$\delta_P$ is the modular character of $P$ and 
$\eta$ is the product of the central
characters of $\pi_1$ and $\pi_2$.  
In this integral the dimension of the group is $n^2-1$.  As for the representations, 
the orbit of the Eisenstein series $E$ is $(21^{n-2})$, so it is of dimension $n-1$ (compare \eqref{dim-of-ES}).  The dimension equation is the equality
$$n^2-1=2(n(n-1)/2)+n-1.$$

If $k\neq n$, suppose $k<n$ without loss.  When $k=n-1$ the integral is a direct generalization of \eqref{Hecke}:
$$\int_{[GL_{n-1}]}\varphi_1\begin{pmatrix}g&\cr&1\end{pmatrix}\varphi_2(g)|\det g|^s\,dg.$$
Here the group is of dimension $(n-1)^2$ and the representations $\pi_1,\pi_2$ are of dimension $n(n-1)/2$ and $(n-1)(n-2)/2$ resp. Since
$$(n-1)^2=n(n-1)/2+(n-1)(n-2)/2$$ 
the dimension equation holds.  However, if $k<n-1$ then a similar integral 
$$\int_{[GL_k]}\varphi_1\begin{pmatrix}g&\cr&I_{n-k}\end{pmatrix}\varphi_2(g)|\det g|^s\,dg$$
does not satisfy the dimension equation: the group is of dimension $k^2$ while the representations are of total dimension 
$n(n-1)/2+k(k-1)/2>k^2$.  One way to satisfy the equation is to introduce an additional integration over a group of dimension $(n^2-n-k^2-k)/2$.
And indeed the integral of Jacquet, Piatetski-Shapiro and Shalika
$$\int_{[GL_k]}\int_{[Y_{n,k}]}\varphi_1\left(y\begin{pmatrix}g&\cr&I_{n-k}\end{pmatrix}\right)\varphi_2(g)\psi(y)|\det g|^s\,dg$$
 includes an integration over the group
$Y_{n,k}$ consisting of upper triangular $n\times n$ unipotent matrices whose upper left $(k+1)\times (k+1)$ corner is the
identity matrix, a group which has dimension $n(n-1)/2-k(k+1)/2$.  To be sure, the integral of \cite{J-PS-Sh}
requires an additive character $\psi$ of this group (the restriction of the standard generic character to $Y_{n,k}$), and this finer level of detail is not seen by the dimension equation,
but the dimension equation already makes the introduction of the group $Y_{n,k}$ natural.

Another example of the dimension equation is practically tautological.  If $E(g,s)$ is an Eisenstein series on a reductive group $G$ formed from generic inducing data, then a straightforward
calculation shows that it too is generic.  If $E(g,s)$ is generic with respect to the
unipotent subgroup $N$ and character $\psi_N$, then its Whittaker coefficient
$$\int_{[N]} E(n,s) \psi_N(n)\, dn$$
satisfies the dimension equation, since both the group and the representation have dimension $\dim(N)$. The Langlands-Shahidi method studies such coefficients
systematically and uses them to get information about $L$-functions.  See Shahidi \cite{Shahidi} and the references there.

Our next example of the dimension equation is given by the integral representing the Asai $L$-function for $GL_n$ (Flicker \cite{Flicker}).  
Suppose $K/F$ is a quadratic extension, and for this example let $\A$ denote the adeles of $F$ and $\A_K$ the adeles of $K$.
Suppose that $\Pi$ is an irreducible cuspidal automorphic representation on $GL_n(\A_K)$, and $\varphi\in\Pi$. Then this integral is of the form
$$\int_{Z(\A) GL_n(F)\backslash GL_n(\A)} \varphi(g) E(g,s)\,dg,$$
where $E(g,s)$ is again a mirabolic Eisenstein series on $GL_n(\A)$ (the central character of $\Pi$ is built into $E$ so the product is invariant under the center $Z(\A)$ of $GL_n(\A))$.
Here the group has dimension $n^2-1$.  This time the automorphic form is on the $\A$-points of the restriction of scalars $\text{Res}_{K/F}GL_n$. Thus the 
dimension of $\pi$ viewed as an automorphic representation over $F$ is twice its dimension over $K$.  That is, in this context $\dim(\Pi)=2\times \frac{n(n-1)}{2}$. The dimension equation is satisfied
since
$$n^2-1=2\times \frac{n(n-1)}{2} + n-1.$$

There are also integrals that represent the product of two different Langlands $L$-functions
{\sl in two separate complex variables}.  These also satisfy the dimension equation. We give a number of examples.
The first instance of such an integral, representing the product of the
standard and exterior square $L$-functions, was provided by Bump and Friedberg \cite{BF}. Suppose
that $\pi$ is a cuspidal automorphic representation of $GL_n$.  If $n=2k$ is even, then the integral is of the form
$$\int_{[GL_1\backslash (GL_k\times GL_k)]}\varphi(\iota(g_1,g_2)) E(g_1,s_1) \left|\frac{\det g_2}{\det g_1}\right|^{s_2-1/2} \,dg_1\,dg_2,$$
where $\iota:GL_k\times GL_k\to GL_{2k}$ is a certain embedding, $\varphi$ is in $\pi$, and $E$ is the mirabolic Eisenstein series
on $GL_k$.  In the quotient $GL_1\backslash (GL_k\times GL_k)$, the group $GL_1$ acts diagonally. This integral satisfies the dimension equation since the group is of dimension
$2k^2-1$ and the representations are of dimensions $(2k)(2k-1)/2$ and $k-1$.  If $n=2k+1$ is odd, then the integral is similar, but is now over
$[GL_1\backslash (GL_{k+1}\times GL_k)]$ with the Eisenstein series on $GL_{k+1}$.  The dimension equation is satisfied in the case that $n$ is odd since
$$(k+1)^2+k^2-1=(2k+1)(2k)/2 + ((k+1)-1).$$

Three additional examples of Rankin-Selberg integrals representing the product of Langlands $L$-functions attached to the standard and spin representations
in separate complex variables
are found in Bump, Friedberg and Ginzburg \cite{BFG-2}. Suppose $\pi$ on $GSp_4$ is generic with trivial central character and $\varphi$ is in the space
of $\pi$.  Let  $P$, $Q$
be the two non-conjugate standard maximal parabolic subgroups of $GSp_4$ and let $E_P$, $E_Q$ be the Eisenstein series induced from
from the modular functions $\delta_P^{s_1}$, $\delta_Q^{s_2}$ resp. Then the integral is of the form
$$\int_{[GL_1\backslash GSp_4]} \varphi(g) E_P(g,s_1) E_Q(g,s_2)\,dg.$$
Here the group is of dimension $10$ and the representations are of dimensions $4,3,3$ resp., so the dimension equation is satisfied.

Suppose next that $\pi$ is on $GSp_6$, generic, and has trivial central character, and let $H$ be the subgroup
of $GL_2\times GSp_4$ of pairs of group elements $(h_1,h_2)$ with equal similitude factors.  Then, for a certain embedding $\iota:H\to GSp_6$, the integral given there is of the form
$$\int_{[GL_1\backslash H]}\varphi(\iota(h_1,h_2)) E_1(h_1,s_1) E_2(h_2,s_2)\,dh_1\,dh_2$$
where $E_1$ (resp.\ $E_2$) is a Borel (resp.\ Siegel) Eisenstein series on $GL_2$ (resp.\ $GSp_4$).  
Here the group is of dimension $13$ and the three representations in the integral
are of dimensions 9, 1, 3 (resp.), as the Siegel Eisenstein series has unipotent orbit $(2^2)$.  

Third suppose that $\pi$ is on $GSp_8$ and is once again generic.  Introduce the subgroup $H$ of $GL_2\times GSp_6$ of pairs with equal similitude factors
and let $\iota:H\to GSp_8$ be a certain embedding given in \cite{BFG}.
Then the authors and Bump prove that the integral
$$\int_{[GL_1\backslash H]} \varphi(\iota(h_1,h_2)) E(h_2,s_1,s_2)\,dh_1\,dh_2$$
is once again a product of two different Langlands $L$-functions, where $E$ is the two-variable Eisenstein series on $GSp_6$ induced from  the 
standard parabolic with Levi factor $GL_1\times GL_2$.  The group has dimension 24, $\pi$ has dimension 16, and indeed $E$ has dimension 8.

To present one further example,
Pollack and Shah \cite{Po-Shah} give an integral representing the product of three Langlands L-functions in three distinct complex variables.  If $\pi$ is on $PGL_4$, then this is of the form
$$\int_{[PGL_4]}\varphi(g) E_1(g,w)E_2(g,s_1,s_2)\,dg$$
where $E_1$ and $E_2$ are the Eisenstein series with Levi factors $GL_2\times GL_2$ and $GL_1^2\times GL_2$, resp.  Here the dimension of the group in the integral
is 15 and the representations are of dimensions 6,4,5 resp.  Those authors also present a related integral on $GU(2,2)$ and the same dimension count holds.

The dimension equation holds as well when covering groups are involved.  For example, the symmetric square integrals of Bump-Ginzburg \cite{Bump-G} and Takeda \cite{Takeda} may be checked
to satisfy this.

Nonetheless, not all Rankin-Selberg integrals satisfy the dimension equation in the form presented in \cite{Ginzburg-2014,Ginzburg-2016}.  
Here is a simple, striking example.
Suppose $\pi$ is a cuspidal automorphic representation of $PGL_4(\A)$, $\varphi$ is in the space of $\pi$, and $E(g,s)$ is an Eisenstein series on
the group $PGSp_4(\A)$.  Consider the integral
\begin{equation}\label{doubleex}\int_{[PGSp_4]} \varphi(g) E(g,s)\,dg.\end{equation}
The basic dimension equation states that
$$ \dim(\pi)+\dim(E)=\dim(PGSp_4)=10.$$
Since $\pi$ is generic, it is of dimension 6, so this equation requires an Eisenstein series of dimension 4, that is, a generic Eisenstein series.  
In fact there is an Eulerian integral with this data.  Using the $GSp_4$ Eisenstein series with trivial central character
induced from an automorphic representation $\tau$ on $GL_2$
via the parabolic with Levi $GL_1\times GL_2$ (the Klingen parabolic), the integral unfolds to the degree 12
$L$-function $L(s,\pi\times \tau,\wedge^2\times \text{standard})$.  Indeed, after using low rank isogenies one sees that this integral is essentially the same as the 
integral for the $SO_n\times GL_k$ standard $L$-function obtained by Ginzburg \cite{Gi-1990} in the case $n=6$, $k=2$.  For these parameters, the $L$-function
is realized as an integral of an $SO_6$ automorphic form against an Eisenstein series on $SO_5$.

However, there is {\sl another}  Rankin-Selberg integral of the form \eqref{doubleex} that is also Eulerian!
That is the case that $E$ is an Eisenstein series of dimension 3 induced from the modular function $\delta_P^s$ of the Siegel parabolic $P$ of $Sp_4$ (one may also twist by
a character of $GL_1$).
Indeed, after using low rank isogenies to again regard this as an integral of an $SO_6$ automorphic form against an Eisenstein series on 
$SO_5$, this is the integral treated by the authors and Bump \cite{BFG} with $n=2,m=1$ (see (1.4) there).  The integral is zero unless 
the automorphic representation $\pi$ is a lift from $Sp_4$, and in that case it represents a degree 5 $L$-function.  However
the dimension equation does not hold, since the Siegel Eisenstein series is of dimension 3, not 4.

There are other important examples of integrals that represent $L$-functions but that do not satisfy the dimension equation of \cite{Ginzburg-2014}.
One class of such integrals are the doubling integrals, a class first constructed by Piatetski-Shapiro and Rallis \cite{PS-R-doubling}, which represent the standard $L$-function of an automorphic
representation on a classical group $G$, twisted by $GL_1$.  The original doubling integral was for split classical groups but the construction has been extended or used by a number of authors including
Yamana \cite{Yam}, Gan \cite{Gan-doubling}, Lapid and Rallis \cite{L-R}, and Ginzburg and Hundley \cite{G-H}.  
The doubling construction was extended to the tensor product $L$-function for $G\times GL_k$ by Cai, Friedberg, Kaplan and Ginzburg \cite{CFGK}. 
(This extension is particularly helpful as it allows one to use the converse theorem to prove lifting results.)
Another class of integrals representing $L$-functions that does not satisfy the original dimension equation is the ``new way" integrals, a class also initiated by Piatetski-Shapiro and Rallis \cite{PS-R-newway}.
Once again this class has been extended, in particular there are the integrals of Bump, Furusawa and Ginzburg \cite{BFuG} and of Gurevich-Segal \cite{Gu-Se}.
A third type of integral that is outside the scope of the dimension equation is the Godement-Jacquet integral \cite{Godo-Jaq}.
Finally, the WO-model integrals of \cite{BFG} do not satisfy the dimension equation.

These last examples all raise the question: is it possible to modify the dimension equation so that it encompasses these examples?  In fact, the answer is yes.   
We explain this in the next Section below.

\section{Extending the Dimension Equation}\label{extending-dc}
Let us begin with the general form of a Rankin-Selberg integral, following \cite{Ginzburg-2016}.  Let $G_i$, $1\leq i\leq l$ be reductive
groups defined over $F$, $\pi_i$ be irreducible automorphic representations of $G_i(\A)$, and let $\varphi_i\in \pi_i$ be automorphic forms.  
Suppose that at least one automorphic form, say $\varphi_1\in \pi_1$, is cuspidal so that the integral to be considered converges.
Let
$U_i\subset G_i$ be subgroups attached to {\it some} unipotent orbits of $G_i$, and let $\psi_i$ be characters of $[U_i]$.
Here the groups $U_i$ may be trivial, and the characters $\psi_i$ need not be in `general position' for $U_i$; 
in particular we do not assume that $U_i$ is attached to the unipotent orbit of $\pi_i$.
We suppose that there is a reductive group $G$ such that for each $i$, the stabilizer of $\psi_i$ inside a suitable Levi subgroup of $G_i$ contains
$G$ (up to isomorphism).  In this case if $\varphi_i$ is an automorphic form on $G_i$ then the Fourier coefficient $\varphi_i^{U_i,\psi_i}$
is automorphic as a function of $G(\A)$.
 Then we consider the integral
\begin{equation}\label{generalRS-DG}\int_{[Z\backslash G]} \varphi_1^{U_1,\psi_1}(g)\varphi_2^{U_2,\psi_2}(g)\dots \varphi_l^{U_l,\psi_l}(g)\,dg,\end{equation}
where $Z$ is the center of $G$ and the central characters are chosen compatibly so that the integrand is $Z(\A)$ invariant.

We are concerned with the case that one automorphic form, say $\varphi_l$, is an Eisenstein series induced from a Levi subgroup $P=MN$ and an automorphic 
representation $\tau$ of $M(\A)$. Write $\varphi_l$ as $\varphi_l(g,s)$ and write the associated
section $f_\tau(g,s)$, so for $\Re(s)\gg0$ we have
$$\varphi_l(g,s)=\sum_{\gamma\in P(F)\backslash G(F)}f_\tau(\gamma g,s).$$
Then the integral \eqref{generalRS-DG} is a function of a complex variable $s$, defined for $\Re(s)\gg0$ and with analytic continuation (possibly with poles) and functional equation by
virtue of the corresponding properties for the Eisenstein series $\varphi_l(g,s)$.

Such integrals may represent $L$-functions when they can be shown to be equal to adelic integrals of some factorizable coefficients of the functions in the integrand.  To write this generally,
suppose that for each $i$, $1\leq i\leq l-1$, the automorphic representation $\pi_i$ attached to $\varphi_i$
has the property that $\mathcal{O}(\pi_i)$ consists of a single unipotent orbit $\mathcal{O}_i$ with unipotent group $V_i$, and that $\varphi_i$ has nonzero Fourier coefficients with
respect to the generic character $\psi_{V_i}$ of $[V_i]$. Write the associated Fourier coefficient
\begin{equation}\label{FC1}L_i(\varphi_i,g)=\int_{[V_i]}\varphi_i(vg)\psi_{V_i}(v)\,dv,\qquad g\in G_i(\A).\end{equation}

Since the first step in analyzing an integral of the form \eqref{generalRS-DG} is to unfold the 
Eisenstein series $\varphi_l(g,s)$, let $V_l=V_\tau\,N$ where $V_\tau$ is the unipotent subgroup of $M\subseteq G_l$ 
corresponding to the maximal
unipotent orbit $\mathcal{O}_\tau$ attached to $\tau$ (again assumed unique), and let $\psi_{V_l}$ be a corresponding character of the unipotent subgroup $V_\tau$ 
extended trivially on $N$. Then we write
\begin{equation}\label{FC2}L_l(f_\tau,g,s)=\int_{[V_l]} f_\tau(vg,s)\psi_{V_l}(v)\,dv,\qquad g\in G_l(\A).\end{equation}

To describe an unfolding of the integral \eqref{generalRS-DG}, for $1\leq i\leq l$, let $R_i$ be a {\sl unipotent} subgroup of $G_i$, and $\psi_{R_i}$ be a character of $R_i(\A)$.
For $1\leq i\leq l-1$ write
$$\varphi_i^{R_i}(g)=\int_{R_i(\A)} L_i(\varphi_i,rg)\psi_{R_i}(r)\,dr,\qquad g\in G_i(\A)$$
and similarly let
$$f_\tau^{R_l}(g,s)=\int_{R_l(\A)} L_l(f_\tau,rg,s)\psi_{R_l}(r)\,dr,\qquad g\in G_l(\A).$$

Then we suppose that for $\Re(s)\gg0$ the integral \eqref{generalRS-DG} unfolds to
\begin{equation}\label{generalRS-unfolded}\int_{Z(\A)M(\A)\backslash G(\A)} \varphi_1^{R_1}(g)\varphi_2^{R_2}(g)\dots \varphi_{l-1}^{R_{l-1}}(g)
f_\tau^{R_l}(w_0g,s)\,dg,\end{equation}
where $M$ is a subgroup of $G$ and $w_0$ is a Weyl group element.
Such an integral is called a {\sl unipotent global integral}.
When the functionals 
$$\mathcal{L}_i(\varphi_i)=\varphi_i^{R_i}(e), 1\leq i\leq l-1;\qquad \mathcal{L}_{l,s}(f_\tau)=f_\tau^{R_l}(e,s)$$
are each factorizable, then the unfolded integral is Eulerian. Such an integral is called a {\sl Eulerian
unipotent integral} in \cite{Ginzburg-2016}.  In fact, this broad class of integrals includes all the Rankin-Selberg integrals presented above 
which satisfy the original form of the dimension equation.  More generally,
the dimension equation
is expected to hold for all Eulerian unipotent integrals, and  a classification of this class of integrals is initiated in \cite{Ginzburg-2016}.

Our goal now is to extend the dimension equation to include many of the examples noted at the last section, that is, Rankin-Selberg integrals
that do not satisfy the dimension equation.  To do so, we begin
with the same integral \eqref{generalRS-DG} but we extend the notation so that each $\varphi_i$ is now either
a single automorphic form in $\pi_i$ or a pair of automorphic forms, one in $\pi_i$ and the other in its contragredient $\widetilde{\pi}_i$. We also
relax the description of the unfolding above in two ways.
We suppose once again that the
integral unfolds to an Eulerian expression \eqref{generalRS-unfolded}.  However, in this expression we no longer assume that $L_i$ is a Fourier coefficient given by an integral 
of the form \eqref{FC1} or \eqref{FC2} over a 
{\sl unipotent} subgroup $V_i$
against a character of the maximal unipotent orbit attached to $\pi_i$ or $\tau$.  Instead we allow the integrals in \eqref{FC1}, \eqref{FC2} to be over arbitrary subgroups of $G_i$.
For example, $L_i$ might be an integral realizing a unique functional such as the Shalika functional
or the spherical functional.  In this case, we do not use $\dim{\pi_i}$ in the dimension equation.  Instead, we replace this term by the dimension
of the full group that realizes the unique functional.  

To be specific, for $1\leq i\leq l-1$ suppose that there is an algebraic group $X_i\subset G_i$, not necessarily unipotent, such that
$$L_i(\varphi_i)(g)=\int_{[X_i]}\varphi_i(xg)\,\psi_{X_i}(x)\,dx,$$
where $\psi_{X_i}$ is a character of $[X_i]$, and let $\mathcal{L}_i$ be the associated functional on $\pi_i$.
In this case we define the dimension of $\mathcal{L}_i$ to be the dimension of the algebraic group $X_i$. For example, if $\varphi_i=(\phi,\phi')$ is a pair of
automorphic functions with $\phi\in \pi_i$, $,\phi'\in\widetilde{\pi}_i$, we may consider $\mathcal{L}_i$ to be the functional that assigns to $\varphi_i$ the global matrix coefficient
\begin{equation}\label{do2}
\mathcal{L}_i(\varphi_i)=\int_{[G_i]}
\phi(g) \phi'(g)\,dg.
\end{equation}
In this case we define $\dim(\mathcal{L}_i)=\dim G_i$.  Similarly, for the Eisenstein series induced from $P$,  we consider 
$$L_l(f_\tau,g,s)=\int_{[X_l]} f_\tau(xg,s)\psi_{X_l}(v)\,dv,\qquad g\in G_l(\A),$$
and we define
$\dim(\mathcal{L}_l)=\dim X_l+\dim N$.
Note that we include the dimension of $N$ in the dimension of the 
functional $\mathcal{L}_l$.  This gives, by definition, an extension of \eqref{dim-of-ES}.

This definition of the dimension of a functional requires a coda.  
To explain why, suppose that $E(g,s)$ is the mirabolic Eisenstein series on $GL_3(\A)$.  This function has unipotent orbit $(21)$ (in fact, it generates the 
automorphic minimal representation for $GL_3(\A)$), so it has dimension 2.  However, if $e_{i,j}$ denotes the matrix with $1$ in position $(i,j)$ and $0$ elsewhere, then it is easy to prove that 
$$\int_{(F\backslash\A)^2} E(I_3+re_{1,2}+me_{1,3},s)\,\psi(r)\,dr\,dm=\int_{(F\backslash\A)^3} E(I_3+re_{1,2}+me_{1,3}+ne_{2,3},s)\,\psi(r)\,dr\,dm\,dn$$
and in fact both corresponding functionals are nonzero and unique.  More generally, when an orbit $\mathcal{O}$ is small the integral defining the Fourier coefficient with respect to that orbit
will have additional invariance properties (there is a nontrivial group that normalizes $U_{\mathcal{O}}$ and stabilizes $\psi_{\mathcal{O}}$) and so can be enlarged in a similar way. Hence to define the dimension of a functional $\mathcal{L}$ 
we must specify that if a functional over a smaller group also realizes $\mathcal{L}$ 
then we use that smaller group in defining the dimension of  $\mathcal{L}$.

We broaden the dimension equation to the following {\sl extended dimension equation}.
\begin{definition}  The Extended Dimension Equation is the equality
\begin{equation}\label{modifed-eqn} \dim(G)+\sum_{i=1}^l \dim(U_i)= \sum_{i=1}^l \dim(\mathcal{L}_i).\end{equation}
\end{definition}
That is, the sum of the dimensions of the groups in the Rankin-Selberg integral \eqref{generalRS-DG} is equal to the sum of dimensions of the {\it functionals} that are obtained
after unfolding.  The key point is that we are using the dimensions of functionals in place of the Gelfand-Kirillov dimensions of the representations $\pi_i$.

We conclude this section by comparing the extended dimension equation with the original form of the dimension equation, 
where $\dim(\mathcal{L}_i)$ is replaced by $\dim(\pi_i)$, a quantity which is computed by \eqref{compute-FJ}.  Suppose first that 
$\pi$ is an automorphic representation such that $\mathcal{O}(\pi)$
is a single unipotent orbit $\mathcal{O}$.  Recall that there is an associated filtration $N_{i,\mathcal{O}}$ of $N$ given by \eqref{filtration}.
Let $\varphi_\pi$ be an automorphic form in the space of $\pi$. 
If $N_{1,\mathcal{O}}=N_{2,\mathcal{O}}$, and if an integral involving $\varphi_\pi$
unfolds to the Fourier coefficient $\varphi_\pi^{U_\mathcal{O},\psi_{\mathcal{O}}}$ i.e.\ to the functional \eqref{FC1} given by integration over $[N_{2,\mathcal{O}}]$, then since
$\dim(\pi)=\dim(N_{2,\mathcal{O}})$ (by \eqref{compute-FJ}), in this situation $\dim{\mathcal{L}}=\dim\pi$. If instead $N_{1,\mathcal{O}}\neq N_{2,\mathcal{O}}$, then 
$N_{1,\mathcal{O}}/ N_{2,\mathcal{O}}$ has the structure of a Heisenberg group.  In this situation, it is often the case that 
an integral involving $\varphi_\pi$ unfolds to a Fourier-Jacobi coefficient of the form
\begin{equation}\label{func1}
{L}(\varphi_\pi)(h):=
\int\limits_{[N_{1,\mathcal{O}}]}\theta_{Sp}(l(v)h)
\varphi_\pi(vh)\,\psi_{N_1}(v)\,dv,
\end{equation}
where $\theta_{Sp}$ is a certain theta function obtained via the Weil representation.  (Here we might need to involve covering groups.)
See for example \cite{GRS-book}, Section~3.2; the notation is given in detail there.  If $\mathcal{L}$ is the functional obtained by composing $L$ with evaluation
at the identity, then in this case we would have 
$\dim{\mathcal L}=\dim{N_{1,\mathcal{O}}}.$
However, if  $\Theta_{Sp}$ is the representation corresponding to $\theta_{Sp}$, then
it is known that $\dim\Theta_{Sp}=\tfrac{1}{2}\dim(N_{1,\mathcal{O}}/ N_{2,\mathcal{O}})$.
Since $\dim\pi=\dim N_{2,\mathcal{O}}+\tfrac12 \dim(N_{1,\mathcal{O}}/ N_{2,\mathcal{O}})$ (see \eqref{compute-FJ}), we obtain
the equality 
$\dim\Theta_{Sp}+\dim\pi=\dim{N_{1,\mathcal{O}}}.$
Thus the definition of $\dim{\mathcal{L}}$
is in fact consistent with original dimension equation in this case. We conclude that when Fourier-Jacobi coefficients are used to construct Eulerian integrals (see for example \cite{GJRS}),
the different versions of the dimension equation we have presented are consistent, and it is accurate to label  \eqref{modifed-eqn} an extension of the original dimension equation.  
Last,  for the Eisenstein series, if the integral in the inducing data $\tau$ unfolds to
the Whittaker functional then the definition of $\dim(\mathcal{L}_l)$ is consistent with \eqref{dim-of-ES} and with the original Dimension Equation.

\section{Examples of the Extended Dimension Equation}\label{examples-of-extend}

In this section we offer examples of the extended dimension equation, and explain how other classes of integrals representing $L$-functions fit into the picture.
To begin, {\it both} integrals of the form \eqref{doubleex} satisfy this extended dimension equation.  Indeed, if $E$ is the Siegel Eisenstein series whose unipotent orbit
has dimension 3, then this integral unfolds to a WO model for $\pi$
in the sense of \cite{BFG}.  This model involves an integration over a 3-dimensional reductive group (a form of $SO_3$) as well as a 4-dimensional unipotent group so in this case the dimension of the functional 
applied to $\pi$ is 7.  
The modified dimension equation does indeed hold, in the form $10=7+3$ (in contrast to the integral involving the Klingen Eisenstein series, where the contributions from
the two functions in the integrand were $6$ and $4$).

Next we discuss doubling integrals.  This is a class of integrals introduced by  Piatetski-Shapiro and Rallis \cite{PS-R-doubling}, of the form
\begin{equation}\label{do1}
\int\limits_{G(F)\times G(F)\backslash G({\A})\times G({\A})}
\varphi_\pi(g)\varphi_\sigma(h)E(\iota(g,h),s)\,dg\,dh
\end{equation}
where $G$ is a symplectic or orthogonal group, $\pi$ and $\sigma$ are two  irreducible cuspidal automorphic representations of $G({\A})$,
and $\varphi_\pi$, $\varphi_\sigma$ are in the corresponding spaces of automorphic forms. The Eisenstein series is defined on 
an auxilliary group $H({\A})$
and $\iota:G\times G \to H$ is an injection.
They show that after unfolding the Eisenstein series, the open orbit representative involves the inner product
\begin{equation}\label{do3}
<\varphi_\pi,\varphi_\sigma>=\int\limits_{G(F)\backslash G({\A})}
\varphi_\pi(g)\varphi_\sigma(g)\,dg
\end{equation}
as inner integration. This integral is nonzero unless $\sigma$ is the contragredient  of $\pi$, and in that case, 
the integral involves the functional \eqref{do2}, that is, the matrix coefficient, which is factorizable.  It is readily checked that
the original form of the dimension equation does not hold.

The dimension of the  functional \eqref{do2} is equal to $\dim(G)$.  Thus the extended dimension equation attached to the integral \eqref{do1} is 
$$2\, \dim(G)=\dim(G)+\dim(E),$$ that is, $$\dim(G)=\dim(E).$$
Moreover, this equation is satisfied by the integrals in \cite{PS-R-doubling}.  For example, 
consider  integral \eqref{do1} with $G=Sp_{2n}$. In this case the Eisenstein series $E(\cdot,s)$ is defined on the group $H({\A})$ with $H=Sp_{4n}$.   It is the maximal parabolic Eisenstein
series attached to the parabolic $P$ with Levi factor $GL_{2n}$ obtained by inducing the modular character $\delta_P^s$.   
Thus $\dim(E)=
\dim(U(P))$, where $U(P)$ is the unipotent radical of $P$. This is given by 
$$\dim(U(P))=\frac{1}{2}(1+2+\cdots+
2n)=n(2n+1)=\dim(Sp_{2n}).$$ 
Thus the extended dimension equation indeed holds for the doubling integral \eqref{do1}.   We remark that $E(\cdot, s)$ is attached to 
the unipotent orbit  $(2^{2n})$.

It may similarly be confirmed that the dimension equation holds for the generalized doubling integrals of Cai, Friedberg, Ginzburg and Kaplan \cite{CFGK},
that represent the Rankin-Selberg $L$-function on $G\times GL_k$, where $G$ is a classical group, attached to the tensor product of the standard representations.
These integrals are of the form
\begin{equation}\label{do5}
\int\limits_{G(F)\times G(F)\backslash G({\A})\times G({\A})}
\varphi_\pi(g)\,\varphi_\sigma(h)\,E_\tau^{U,\psi_U}(\iota(g,h),s)\,dg\,dh
\end{equation}
where now $E_\tau$ is an Eisenstein series on a larger group $H$ whose construction depends on $\tau$ and the superscripts on $E$ denote a Fourier coefficient
with respect to a unipotent group $U$ and character $\psi_U$ such that $\iota(G\times G)$ is contained in the stabilizer of $\psi_U$ inside the normalizer of $U$ in $H$.
Suppose that $G=Sp_{2n}$.  Then $H=Sp_{4kn}$, the Eisenstein series is induced from a generalized Speh represntation on $GL_{2kn}$ which has unipotent
orbit $(k^{2n})$, and 
the $(U,\psi_U)$ coefficient is one corresponding to the unipotent orbit $((2k-1)^{2n}1^{2n})$ in $H$.
The integral once again unfolds to an integral involving matrix coefficients.  
Due to the Fourier coefficient with respect to $(U,\psi_U)$, the extended dimension equation in this case becomes
\begin{equation}\label{dim1}
 2\dim(Sp_{2n}) +\dim(U) = \dim(Sp_{2n}) + \dim(E_{\tau}).
\end{equation}
To show that this is true, it follows from \cite{G1} that
\begin{equation*}
\dim(E_{\tau})= \frac{1}{2}\dim((k)^{2n})
+ \dim(U(P)),
\end{equation*}
where $\dim((k)^{2n})$ is the dimension of the unipotent orbit of the inducing data,
and $U(P)$ is the unipotent radical of the maximal parabolic $P$ inside $H$ whose Levi factor is $GL_{2kn}$.  (That is, \eqref{dim-of-ES} holds in this situation.)
The number $\frac{1}{2}\dim((k)^{2n}))$ is equal to
the dimension of the unipotent radical of the parabolic subgroup of $GL_{2nk}$ whose Levi
part is $GL_{2n}^k$, that is, $2n^2k(k-1)$. It is then easy to check that \eqref{dim1} holds.

There are other classes of Eulerian integrals.  The `new way' integrals of \cite{PS-R-newway} unfold to functionals that are not unique.  They satisfy the extended dimension equation,
albeit tautologically.  The integrals of Godement-Jacquet type may be obtained from doubling integrals after unfolding.  Hence they should be regarded as belonging
to this paradigm.  Whether or not this is helpful for efforts to extend the method (the Braverman-Kazhdan-Ng$\hat{\textrm{o}}$ program)  remains to be seen.

\section{Integral Kernels and the Dimension Equation}\label{integral-kernels}

Integral kernels appear often in the theory of automorphic forms as a way to relate automorphic forms on one group to automorphic
forms on a different group.  In this brief Section, we explain the connection between such constructions and the dimension equation and illustrate the use
of this equation to detect properties of such a correspondence.  We follow \cite{Ginzburg-2014} and provide an additional example of interest.
Then in the next Section we turn to a similar analysis using the extended dimension equation.

Suppose that $G,H,L$ are reductive groups and there is an embedding $\iota:G\times H\to L$ such that the images of $G$ and $H$ in $L$ commute.
Let $U$ be a unipotent subgroup of $L$ and $\psi_U$ be a character of $[U]$ whose stabilizer in $L$ contains $\iota(G(\A),H(\A))$.
Let $\Theta$ denote an
automorphic representation of $L(\A)$.   Then one may seek to construct a lifting from automorphic representations of $G(\A)$ to $H(\A)$ as follows.

 Let $\pi$ denote an irreducible cuspidal representation of $G({\A})$. Let $\sigma$ be the representation of $H({\A})$ generated by all functions of the form 
\begin{equation}\label{lift1}
f(h)=\int\limits_{[G]}\ 
\int\limits_{[U]}
\varphi_\pi(g)\,\theta(u(g,h))\,\psi_U(u)\,du\,dg,
\end{equation}
with $\theta$ in the space of $\Theta$ and $\varphi_\pi$ in the space of $\pi$. Notice that from the properties of $\Theta$, 
the functions $f(h)$ are $H(F)$-invariant functions on $H(\A)$.  As a first case, suppose that $\sigma$ is an irreducible automorphic representation of $H(\A)$.
In that case, the dimension equation attached to this construction is 
\begin{equation}\label{dim3} 
\dim(G)+\dim(U)+ \dim(\sigma)=\dim(\pi)
+\dim(\Theta).
\end{equation}
That is, since the integral \eqref{lift1} gives a representation on $H(\A)$ instead of an $L$-function, we include the dimension of the lift, $\sigma$, with the dimensions of the groups.
See \cite{Ginzburg-2014}, Section 6.  More generally, suppose that $\sigma$ is in the discrete part of the space $L^2(H(F) Z_H(\A)\backslash H(\A),\omega)$, where
$Z_H$ is the center of $H$ and $\omega$ is a central character (this is true, for example, when $\sigma$ is cuspidal). In that case we expect that at least one of the summands of $\sigma$
will satisfy equation \eqref{dim3}. 

This simple equation turns out to be quite powerful. We illustrate with an example.  
Suppose that $G=Sp_{2n}$, $H=SO_{2k}$, $L=Sp_{4nk}$, and $\Theta$ is the classical theta representation.
Note that the unipotent orbit attached to $\Theta$ is $(21^{4nk-2})$, and $\dim(\Theta)=2nk$. Suppose also that $\pi$ is generic, so 
$\dim(\pi)=n^2$. Since $\dim(G)=2n^2+n$,
 the dimension equation \eqref{dim3} becomes
$$\dim(\sigma)=2nk-n-n^2.$$
As a first consequence, if $k<\tfrac{n+1}{2}$ then the equation would assert that $\dim(\sigma)<0$.  So we expect that the lift must be zero.
Second, let us ask for which $k$ the lift $\sigma$ can be generic.  In that case, we would have $\dim(\sigma)=k^2-k$.  We conclude that a necessary condition for $\sigma$ to be
generic is the condition
$$k^2-k=2nk-n-n^2.$$
For a fixed $n$, there are two solutions to this equation, namely $k=n+1$ and $k=n$.  And indeed, both these consequences of the dimension equation are true.  
The lift does vanish if $k<\tfrac{n+1}{2}$. And
the lift to $SO_{2k}$ with $k=n+1$ is always generic while the lift with $k=n$ is sometimes generic, and these are the only cases where the lift of a generic cuspidal automorphic 
representation is generic.   See \cite{GRS-theta}, Cor.\ 2.3 and the last two paragraphs in Section 2;
for the analogous local result see Proposition 2.4 there.

\section{Doubling Integrals and Integral Kernels}\label{general}

In this Section we connect doubling integrals and integral kernels.  First let us describe such doubling integrals in general.
Suppose that $H$ is a group,  $U\subseteq H$ is a unipotent subgroup, and $\psi_U$ is a character of $[U]$. Suppose that there is an embedding $\iota:G\times G\to N_H(U)$ whose
image fixes $\psi_U$ (under conjugation).  Let $\pi$, $\sigma$ be automorphic representations of a group $G(\A)$.
Then 
we consider integrals of the form
\begin{equation}\label{do33}
\int\limits_{[G\times G]}\ 
\int\limits_{[U]}
\varphi_\pi(g)\,\varphi_\sigma(h)\,E(u\iota(g,h),s,f_s)\psi_U(u)\,du\,dg\,dh.
\end{equation}
Note that this includes the integrals \eqref{do1} considered by Piatetski-Shapiro and Rallis but that 
we allow an extra unipotent integration, as in \eqref{do5}.
We  say that the integral \eqref{do33} is a doubling integral if for $\Re(s)$ large it is equal to
\begin{equation}\label{do331}
\int\limits_{  G({\A})}\ 
\int\limits_{ U_0({\A})}
<\varphi_\pi,\sigma(g)\varphi_\sigma>
f_W(\delta u_0(g,1),s)\psi_U(u_0)\,du_0\,dg,
\end{equation}
for some subgroup $U_0$ of $U$,  some Fourier coefficient $f_W$ of the section $f_s$, and some
$\delta\in H(F)$. Due to the matrix coefficient, this integral is zero unless $\sigma$ is the contragredient of $\pi$, so we 
suppose this from now on.

The extended dimension equation attached to the integral \eqref{do33} is
\begin{equation}\label{dim111} 
\dim(G)+\dim(U)=\dim(E).
\end{equation}
We emphasize that if a certain integral satisfies the extended dimension equation \eqref{dim111} it does not necessarily means that the integral will be non-zero or Eulerian. This can only be determined after the unfolding process. The advantage of the equation \eqref{dim111} is that it eliminates unlikely candidates. 

Suppose that \eqref{do33} is a doubling integral. Then we may make an integral kernel as follows. 
Fix a cuspidal automorphic representation $\pi$ of $G(\A)$, and consider the functions
\begin{equation}\label{lift2}
f(h)=\int\limits_{[G]}
\int\limits_{[U]}
\varphi_\pi(g)E(u(g,h),s)\psi_U(u)\,du\,dg.
\end{equation}
Arguing as in Theorem 1 of Ginzburg and Soudry \cite{G-S-2018} we deduce that the representation $\sigma$
generated by these functions is a certain twist of the representation $\pi$. 
In particular we have $\dim(\sigma)=\dim(\pi)$. Hence the dimension equation \eqref{dim3} is the same as the extended dimension equation  \eqref{dim111}.
Accortdingly, one can view the construction given by \eqref{lift1} as a generalization of the construction of doubling integrals. 

This discussion leads to the following Classification Problem, which illustrates the kind of question
that these constructions raise. Let $\pi$ denote a cuspidal representation of $G({\A})$, and $\sigma$ denote a cuspidal representation of $H({\A})$. 
Find examples of representations $\Theta$ defined on a group $L({\A})$ as above which satisfy the following two conditions:

\begin{enumerate}
\item The dimension equation \eqref{dim3} holds.
\item Suppose that the integral  
\begin{equation}\label{lift3}
\int\limits_{[G\times H]}
\int\limits_{[U]}
\varphi_\sigma(h)\varphi_\pi(g)\theta(u(g,h))\,\psi_U(u)\,du\,dg\,dh
\end{equation}
is not zero for some choice of data. Then the representation $\pi$ determines the representation $\sigma$ uniquely.
\end{enumerate}
This classification problem has many solutions as stated.  We present an example of how it may be approached in the next Section.

We remark that a similar analysis based on the dimension equation may be applied to the descent integrals of Ginzburg, Rallis and Soudry \cite{GRS-book}.

\section{An Example}\label{example}
In this Section we will show how the above considerations can help find possible global integrals of the form given by equation 
\eqref{lift3}.  To do so, in this Section we will work with the case $G=H=Sp_{2m}$. 

The first step is to consider Fourier coefficients whose stabilizer contains the group $Sp_{2m}\times Sp_{2m}$. 
From the theory of nilpotent orbits, the partition 
$((2k-1)^{2m}(2r-1)^{2m})$ has this property (see \cite{C-M}). 
This leads us to look for a representation $\Theta$ defined on the group 
$Sp_{4m(k+r-1)}({\A})$ which is $Sp_{4m(k+r-1)}(F)$ invariant and which satisfies the dimension equation given in \eqref{dim3}.
Thus, we are seeking representations $\Theta$ such that $\mathcal{O}(\Theta)=\mathcal{O}$ with 
\begin{equation}\label{dim6}
\dim(Sp_{2m})+\tfrac{1}{2}\dim((2k-1)^{2m}(2r-1)^{2m})=
\tfrac{1}{2}\dim({\mathcal O}).
\end{equation} 
Here $\mathcal{O}$ corresponds to a partition of the number $4m(k+r-1)$.  There are many solutions.
For example, if we begin with the orbit $(5^23^2)$, that is $m=1$, $k=3$ and $r=2$, then the orbits $\mathcal{O}$ equal to
$(65^2)$, $(83^22)$, $(6^22^2),$ and $(8421^2)$ all satisfy condition \eqref{dim6}.

Since this low rank case already offers so many possibilities, 
this suggests that it is not expeditious 
to classify all orbits $\mathcal{O}$ which satisfy \eqref{dim6}. 
Experience suggests that a good place to begin a further analysis is to focus on orbits of the form $(n_1^{2l_1}n_2^{2l_2}\ldots
n_p^{2l_p})$ such that $p$ is minimal. For example, in the case above, if we begin with the orbit $(5^23^2)$, 
we would seek $\Theta$ such that $\mathcal{O}(\Theta)=(6^22^2)$. 
In general we have
\begin{lemma}\label{lem1}
We have 
\begin{equation}\label{dim7}\notag
\dim(Sp_{2m})+\tfrac{1}{2}\dim((2k-1)^{2m}(2r-1)^{2m})=
\tfrac{1}{2}\dim((2k)^{2m}(2r-2)^{2m}).
\end{equation}
\end{lemma}
\begin{proof}
Using equation \eqref{dim2}, one may compute the difference
$$\tfrac{1}{2}\dim((2k)^{2m}(2r-2)^{2m})-
\tfrac{1}{2}\dim((2k-1)^{2m}(2r-1)^{2m})$$
and to show that it is equal to $2m^2+m=\dim(Sp_{2m})$.
We omit the details. 
\end{proof}

We observe that the doubling construction of the authors, Cai and Kaplan \cite{CFGK} fits this rubric, indeed, it provides an example
of such an integral with $r=1$.
More precisely, the representation $\Theta$ used in \cite{CFGK} is an Eisenstein series 
$E(\cdot,s)$  defined on $Sp_{4mk}({\A})$, and it satisfies ${\mathcal O}(E(\cdot,s))=((2k)^{2m})$. 
Then a Fourier coefficient of $E(\cdot,s)$ is taken with respect to a unipotent group $U$ and generic character corresponding to the partition
$((2k-1)^{2m}1^{2m})$ of $4km$.

Returning to the general case, the next step is to find an automorphic representation $\Theta$ defined on the group $Sp_{4m(k+r-1)}({\A})$ 
which satisfies ${\mathcal O}(\Theta)=((2k)^{2m}(2r-2)^{2m})$. The main source of examples for such representations are Eisenstein series and their residues. 
If we have a candidate for $\Theta$ which is an Eisenstein series, then  by unfolding the Eisenstein series it is often
possible to check if the non-vanishing of 
equation \eqref{lift3} implies that $\pi$ determines 
$\sigma$ uniquely. See \cite{CFGK} for an example. However, if the representation $\Theta$ is obtained as a {\sl residue} of an Eisenstein series, an unfolding process is 
not readily available to us unless we attempt to unfold first and then take the residue, a strategy that is often problematic, and in this case, typically only a weaker statement can be checked. 

In the rest of this Section we will describe a simple case where the representation $\Theta$ is not an Eisenstein series.  We will only give the flavor 
of the construction here; we plan to present 
the details in a
separate paper.

Let $\tau$ denote an irreducible cuspidal representation of the group $GL_n({\A})$.  Suppose that $n>2$ is even and that the partial $L$-function 
$L^S(\tau,\wedge^2,s)$ has a simple pole at $s=1$. Let ${\mathcal E}_\tau$ denote the generalized Speh representation defined on the group 
$GL_{3n}({\A})$. See \cite{CFGK}.  Let $P(GL_{3n})$ be the maximal parabolic subgroup of $Sp_{6n}$ whose Levi part is $GL_{3n}$ (the so-called
Siegel parabolic).
Let $E_\tau(\cdot,s)$ denote the Eisenstein series defined on the group $Sp_{6n}({\A})$ attached to the induced space $Ind_{P(GL_{3n})({\A})}^{Sp_{6n}({\A})}
{\mathcal E}_\tau\delta_P^s$. The poles of this Eisenstein series are determined by the poles of $L^S(\tau,\wedge^2,s)$ and $L^S(\tau,\vee^2,s)$. See \cite{J-L-Z}. It follows from that reference that if $L^S(\tau,\wedge^2,s)$ has a simple pole at $s=1$, then the Eisenstein series $E_\tau(\cdot,s)$ has a simple pole at $s_0=(3n+2)/(6n+2)$. Denote $\Theta'_\tau=\text{Res}_{s=s_0}
E_\tau(\cdot,s)$.  That is, $\Theta'_\tau$ is the automorphic representation spanned by the residues of this family of Eisenstein series at the point $s_0$.
Then, it follows from \cite{G-S} that  there is an irreducible constituent of $\Theta_\tau$ of $\Theta'_\tau$ such that
\begin{equation}\label{conj10}
{\mathcal O}(\Theta_\tau)=((2n)^2n^2).
\end{equation}
Choosing $k=n$, $m=1$ and $r=(n+2)/2$ in Lemma \ref{lem1} we conclude that the dimension equation \eqref{dim3} is satisfied,
and we may form the global integral \eqref{lift3}.

More generally, suppose that the representation $\Theta$ in \eqref{lift3} satisfies ${\mathcal O}(\Theta)=((2n)^2n^2)$.  
Suppose moreover that $\Theta=\otimes'_\nu \Theta_\nu$  (restricted tensor product) where for almost all $\nu$ the unipotent orbit attached to $(\Theta)_\nu$ is also $((2n)^2n^2)$.
This is true if we take $\Theta$ to
be $\Theta_\tau$. Indeed, in this case the corresponding local statement may be proved by arguing
analogously to the global case, i.e.\ replacing Fourier expansions by the geometric lemma, global root
exchange by local root exchange, etc. 

Suppose that $\pi$ is an irreducible cuspidal automorphic representation of the group $SL_2(\A)$ whose image under the `lift' corresponding to $\Theta$
contains an irreducible cuspidal representation $\sigma$, also on the group $SL_2({\A})$.
We seek to establish the relation between $\pi$ and $\sigma$ assuming that integral \eqref{lift3} is not zero for some choice of data. 
Since the representation $\Theta$ is not an Eisenstein series we cannot simply carry out an unfolding process. However, in this case we can show that the representations $\pi$ and $\sigma$ are nearly equivalent. 
We sketch a local argument for representations in general position, omitting the details.

Suppose that $\pi=\otimes'_\nu\pi_\nu$,  $\sigma=\otimes'_\nu\sigma_\nu$. At unramified places suppose that 
$\pi_\nu=\Ind_B^{SL_2}\chi_\nu\delta_B^{1/2}$ and that $\sigma_\nu=\Ind_B^{SL_2}\mu_\nu\delta_B^{1/2}$,
where $B$ is the standard Borel subgroup of $SL_2$ and $\chi_\nu$ and $\mu_\nu$ are unramified characters. Suppose that 
$\chi_\nu$ is in general position. If the integral \eqref{lift3} is not zero for some choice of data, then the space
\begin{equation}\label{hom1.0}
\Hom_{SL_2\times SL_2}(\Ind_B^{SL_2}\chi_\nu\delta_B^{1/2}\times 
\Ind_B^{SL_2}\mu_\nu\delta_B^{1/2},J_{U,\psi_U}(\Theta_\nu))
\end{equation}
is not zero. Here $J_{U,\psi_U}$ is the local twisted Jacquet module 
which corresponds to the Fourier coefficient over the group $U$ and the character $\psi_U$ 
of integral \eqref{lift3}. It follows from Frobenius reciprocity that the space \eqref{hom1.0} is equal to
\begin{equation}\label{hom2}
\Hom_{GL_1\times GL_1}(\chi_\nu\delta_B^{1/2} 
\mu_\nu\delta_B^{1/2},J_{U_1,\psi_U}(\Theta_\nu)).
\end{equation}
Here $U_1$ is a certain unipotent subgroup which contains $U$ and the character $\psi_U$ is the trivial extension from $U$ to $U_1$. 
Then performing root exchanges and using that the unipotent orbit attached to $(\Theta)_\nu$ is $((2n)^2n^2)$, 
one can prove 
that the nonvanishing of the space \eqref{hom2} implies that the space
\begin{equation}\label{hom3}
\Hom_{GL_1}(\chi_\nu\delta_B^{1/2} 
\mu_\nu\delta_B^{1/2},J_{U_2,\psi_{U_2}}(\Theta_\nu))
\end{equation}
is also nonzero. Here $GL_1$ is embedded in 
$GL_1\times GL_1$ diagonally, and $J_{U_2,\psi_{U_2}}$ is the twisted Jacquet module attached to a certain unipotent group
$U_2$ and character $\psi_{U_2}$. This last space may then be analyzed by using properties of the representation $\Theta_\nu$,
and to deduce that $\mu_\nu=\chi_\nu^{\pm 1}$.
This example illustrates how the dimension equation may be used to suggest new integral kernels.

\section{Local Analogues} \label{local-section}

The unfolding of a Rankin-Selberg integral typically has a local analogue, so 
when the dimension equation or extended dimension equation is satisfied, it is natural to seek local statements that hold.  Once again we view the equation as necessary but not sufficient.
In this Section we illustrate the local statements that arise.  Let $K$ be a non-archimedean local field whose residue field has cardinality $q$.

A first example is given by the Rankin-Selberg integrals of Jacquet, Piatetski-Shapiro and Shalika \cite{J-PS-Sh}.
Let $\pi_1$, $\pi_2$ be irreducible admissible generic representations of $GL_n(K)$, $GL_k(K)$, resp.
We keep the notation of Section~\ref{dim-eq-examples} above.
Consider $\pi_2$ to be a module for $Y_{n,k}(K)$ via $\psi$.
If $k<n$, then (\cite{J-PS-Sh}, paragraph (2.11), Proposition) the space of $(GL_k\ltimes Y_{n,k})(K)$-equivariant bilinear forms
$$\text{Bil}_{(GL_k\ltimes Y_{n,k})(K)}(\pi\otimes |\det(\cdot)|^s,\pi')$$
has dimension at most 1, except for finitely many values of $q^{-s}$.  
Similarly, if $n=k$, 
the space
$$\text{Bil}_{GL_n(K)}(\pi\otimes\pi'\otimes |\det(\cdot)|^s,\Ind_{P(K)}^{GL_n(K)}\delta_{P}^{-1/2})$$
has dimension at most 1, except for finitely many values of $q^{-s}$
(\cite{J-PS-Sh}, paragraph (2.10), Eqn.\ (5); see also the Proposition in (2.10) for an equivalent formulation in terms of trilinear forms).
We caution the reader that in the literature this is presented using $GL_n(K)$-equivariance rather than $PGL_n(K)$-equivariance, but unless the central characters are chosen compatibly the
space is zero.  We need to use $PGL_n$ to satisfy the dimension equation. This is comparable to insisting that we choose the group of smallest possible dimension in assigning a dimension to a functional.

As an additional example, in the situation of the generalized doubling integrals of Cai, Kaplan and the authors \cite{CFGK}, the study of the global integral \eqref{do5} leads to the local Hom space
$$\text{Hom}_{G(K)\times G(K)}(J_{U,\psi_U^{-1}}(\Ind_{P(K)}^{H(K)}(W_c(\tau)\delta_P^s),\pi^\vee \otimes \pi)$$
where $J_{U,\psi_U^{-1}}$ denotes a twisted Jacquet module with respect to the group $U(K)$ and character $\psi_U^{-1}$ and the remaining notation is given in \cite{CFGK}, see especially (3.2) there.
Once again, this space is at most one dimensional except for finitely many values of $q^{-s}$
(see the proof of \cite{CFGK}, Theorem 21).
As explained above, a dimension equation holds
but only if we use the extended dimension equation and treat $\pi^\vee \otimes \pi$ as having dimension equal to $\dim(G)$.

Finally, when the dimension equation appears in the context of a lifting result, then one may hope to prove the existence of a local correspondence similar to the Howe correspondence for the
classical theta representation.  The local concerns that arise are illustrated by the treatment of \eqref{hom1.0} above.  In particular, it is natural to seek to extend
such a correspondence beyond a matching of the unramified principal series.


\begin{thebibliography}{AAAAA}

\bibitem{B-V} Barbasch, Dan: Vogan, David A. Unipotent representations of complex semisimple groups.  Annals  of Math. 121 (1) (1985), 41--110.

\bibitem{Bump-2005} Bump, Daniel.
The Rankin-Selberg method: an introduction and survey. In: Automorphic representations, L-functions and applications: progress and prospects, 41--73, 
Ohio State Univ. Math. Res. Inst. Publ., 11, de Gruyter, Berlin, 2005.

\bibitem{BF} Bump, Daniel; Friedberg, Solomon.
The exterior square automorphic L-functions on GL(n). Festschrift in honor of I. I. Piatetski-Shapiro on the occasion of his sixtieth birthday, Part II (Ramat Aviv, 1989), 47--65, 
Israel Math. Conf. Proc., 3, Weizmann, Jerusalem, 1990

\bibitem{BFG} Bump, Daniel; Friedberg, Solomon; Ginzburg, David. Whittaker-orthogonal models, functoriality and the Rankin-Selberg method.
Invent. Math. 109 (1992), 55-96.

\bibitem{BFG-2} Bump, Daniel; Friedberg, Solomon; Ginzburg, David. 
Rankin-Selberg integrals in two complex variables. 
Math. Ann. 313 (1999), no. 4, 731--761.

\bibitem{BFG-theta} Bump, Daniel; Friedberg, Solomon; Ginzburg, David.  Small representations for odd orthogonal groups.
IMRN 2003 (2003), no.\ 25, 1363--1393.

\bibitem{BFG-2006} Bump, Daniel; Friedberg, Solomon; Ginzburg, David. 
Lifting automorphic representations on the double covers of orthogonal groups. 
Duke Math. J. 131 (2006), no. 2, 363--396.

\bibitem{BFuG} Bump, Daniel; Furusawa, Masaaki; Ginzburg, David. Nonunique models in the Rankin-Selberg method. J. reine angew. Math. 468 (1995), 77--111.

\bibitem{Bump-G} 
Bump, Daniel; Ginzburg, David.
Symmetric square L-functions on $GL(r)$.
Annals of Math. (2) 136 (1992), no. 1, 137--205.

\bibitem{Cai} Cai, Yuanqing.  Fourier coefficients for degenerate Eisenstein series
and the descending decomposition. Manuscripta math. 156 (2018), 69--501.

\bibitem{CFGK} Cai, Yuanqing; Friedberg, Solomon; Ginzburg, David; Kaplan, Eyal.  Doubling constructions and tensor product L-functions: the linear case. To appear in
Inventiones Math.;
arXiv:1710.00905.

\bibitem{C-M} Collingwood, David H.; McGovern, William M.
Nilpotent orbits in semisimple Lie algebras. Van Nostrand Reinhold
Mathematics Series. Van Nostrand Reinhold Co., New York, 1993.

\bibitem{Flicker} Flicker, Yuval Z.
Twisted tensors and Euler products.
Bull. Soc. Math. France 116 (1988), no. 3, 295--313.

\bibitem{Gan-doubling} Gan, Wee Teck.  Doubling zeta integrals and local factors for metaplectic groups. 
Nagoya Math. J. 208 (2012), 67--95.

\bibitem{Gan-ICM} Gan, Wee Teck. Theta correspondence: recent progress and applications. 
Proceedings of the International Congress of Mathematicians---Seoul 2014. Vol. II, 343--366, Kyung Moon Sa, Seoul, 2014.

\bibitem{Gi-1990} Ginzburg, David.  $L$-functions for $SO_n\times GL_k$. J. Reine Angew. Math. 405 (1990), 156--180.

\bibitem{G1} Ginzburg, David. Certain conjectures relating unipotent orbits to automorphic representations. Israel J. Math. 151 (2006), 323--355.

\bibitem{Ginzburg-2014} Ginzburg, David. Towards a classification of global integral constructions and functorial liftings using the small representations method.
Advances in Math. 254 (2014), 157--186.

\bibitem{Ginzburg-2016} Ginzburg, David.
Classification of some global integrals related to groups of type $A_n$.
J. Number Theory 165 (2016), 169--202.

\bibitem{G-H} Ginzburg, David; Hundley, Joseph.
A doubling integral for G2. (English summary) 
Israel J. Math. 207 (2015), no. 2, 835--879.

\bibitem{GJRS} Ginzburg, David; Jiang, Dihua; Rallis, Stephen; Soudry, David.
L-functions for symplectic groups using Fourier-Jacobi modelsArithmetic geometry and automorphic forms, 183--207,
Adv. Lect. Math. (ALM), 19, Int. Press, Somerville, MA, 2011.

\bibitem{GRS-theta} Ginzburg, David; Rallis, Stephen; Soudry, David.  Periods, poles of $L$-functions and symplectic-orthogonal theta lifts. J. Reine Angew. Math. 487 (1997), 85--114.

\bibitem{G-R-S} Ginzburg, D.; Rallis, S.; Soudry, D. On Fourier coefficients of automorphic forms of symplectic groups. Manuscripta Math. 111 (2003), no. 1, 1--16.

\bibitem{GRS-book} Ginzburg, David; Rallis, Stephen; Soudry, David.  The Descent Map from Automorphic Representations of $GL(n)$ to Classical Groups.
World Scientific Publishing Co. Pte. Ltd., Hackensack, NJ, 2011.

\bibitem{G-S-2018} Ginzburg, David; Soudry, David. Integrals derived from the doubling method.
Preprint, 2018.  arXiv:1810.08913.

\bibitem{G-S} Ginzburg, David; Soudry, David. In preparation. 

\bibitem{Godo-Jaq} Godement, Roger; Jacquet, Herv\'e.
Zeta functions of simple algebras. 
Lecture Notes in Mathematics, Vol. 260. Springer-Verlag, Berlin-New York, 1972.

\bibitem{Gu-Se} Gurevich, Nadya; Segal, Avner.
The Rankin-Selberg integral with a non-unique model for the standard $\mathcal{L}$-function of $G_2$.
J. Inst. Math. Jussieu 14 (2015), no. 1, 149--184.

\bibitem{J-PS-Sh} Jacquet, Herv\'e; Piatetskii-Shapiro, Ilya; Shalika, Joseph. 
Rankin-Selberg convolutions. 
Amer. J. Math. 105 (1983), no. 2, 367--464.

\bibitem{J-L-Z} Jiang, Dihua; Liu, Baiying; Zhang, Lei.
Poles of certain residual Eisenstein series of classical groups. Pacific J. Math. 264 (2013), no. 1, 83--123.

\bibitem{K} Kawanaka, Noriaki.
Shintani lifting and Gel'fand-Graev representations. The Arcata Conference on Representations of Finite Groups (Arcata, Calif., 1986), 147--163,
Proc. Sympos. Pure Math., 47, Part 1, Amer. Math. Soc., Providence, RI, 1987.

\bibitem{L-R} Lapid, Erez M.; Rallis, Stephen.
On the local factors of representations of classical groups.In: Automorphic representations, L-functions and applications: progress and prospects, 309--359, 
Ohio State Univ. Math. Res. Inst. Publ., 11, de Gruyter, Berlin, 2005.

\bibitem{Leslie} Leslie, Spencer. 
A Generalized Theta lifting, CAP representations, and Arthur parameters. 
Preprint, arXiv:1703.02597.

\bibitem{PS-R-doubling} Piatetski-Shapiro, Ilya; Rallis, Stephen. L -functions for the classical groups. In: Gelbart, Stephen; Piatetski-Shapiro, Ilya; Rallis, Stephen.
Explicit constructions of automorphic L-functions. 
Lecture Notes in Mathematics, 1254. Springer-Verlag, Berlin, 1987.

\bibitem{PS-R-newway} Piatetski-Shapiro, I.; Rallis, S.
A new way to get Euler products. 
J. Reine Angew. Math. 392 (1988), 110--124.

\bibitem{Po-Shah} Pollack, Aaron; Shah, Shrenik.
Multivariate Rankin-Selberg integrals on $GL_4$ and $GU(2,2)$.  
Canad. Math. Bull. 61 (2018), no. 4, 822--835.

\bibitem{Shahidi} Shahidi, Freydoon. 
Eisenstein series and automorphic L-functions. 
American Mathematical Society Colloquium Publications, 58. 
American Mathematical Society, Providence, RI, 2010. 

\bibitem{Takeda} Takeda, Shuichiro.
The twisted symmetric square L-function of $GL(r)$. 
Duke Math. J. 163 (2014), no. 1, 175--266.

\bibitem{Yam} Yamana, Shunsuke.
L-functions and theta correspondence for classical groups. 
Invent. Math. 196 (2014), no. 3, 651--732.

\end{thebibliography}
\end{document}